\documentclass[12pt]{amsart}

\usepackage[utf8]{inputenc}
\usepackage{amsmath}
\usepackage{amssymb}
\usepackage{amsthm}
\usepackage{parskip}
\usepackage{tikz}
\usepackage{hyperref}
\usepackage{listings}
\usepackage{fullpage}
\usepackage{xcolor}
\usepackage{tabularray}
\usepackage{enumerate}
\usepackage{multirow}
\usepackage{multicol}
\numberwithin{equation}{section}
\newtheorem{theorem}{Theorem}[section]

\newtheorem{corollary}[theorem]{Corollary}

\newtheorem{proposition}[theorem]{Proposition}

\newtheorem{conjecture}[theorem]{Conjecture}

\theoremstyle{remark}

\newtheorem{example}[theorem]{Example}

\newtheorem*{rem}{Remark}

\newcounter{procedureCTR}

\numberwithin{procedureCTR}{section}

\newcommand{\Z}{\mathbb{Z}}
\newcommand{\la}{\lambda}
\newcommand{\de}{\delta}
\newcommand{\B}{\mathcal{B}}
\renewcommand{\O}{\mathcal{O}}
\newcommand{\cyc}[1]{\langle{#1}\rangle}

\definecolor{woodbrown}{RGB}{193,154,107}
\definecolor{deeppurple}{RGB}{128,0,128}

\definecolor{paleYellow}{RGB}{253,255,221}
\definecolor{palePink}{RGB}{254,207,247}
\definecolor{greenish}{RGB}{177,254,120}
\definecolor{forestgreen}{RGB}{41,194,44}

\title{Toward a Canonical Representation of Blocked Rectangular Grids with an Application to Finite Tiling Problems}

\author{Noah Jensen}
\address{Department of Mathematics and Statistics, Portland State University, Portland, OR 97201}
\email{noahj@pdx.edu}

\author{Stephanie Treneer}
\address{Department of Mathematics, Western Washington University, Bellingham, WA 98225}
\email{trenees@wwu.edu}

\begin{document}

\begin{abstract}
 Given the collection of all $m\times n$ rectangular grids which have a fixed number $1\leq r\leq mn$ of blocked cells, we explicitly describe a proper subset of the collection which is guaranteed to contain at least one grid from each equivalence class under symmetry, eliminating the majority of redundant grids. We analyze the extent to which redundant grids remain in the reduced set, and give general cases in which our methods exactly produce a complete set of canonical representatives for the equivalence classes. As an application of our results, we specify collections of polyomino tiling problems and find all solvable grids in each collection.
\end{abstract}

\maketitle

\section{Introduction}

\begin{figure}\label{GSB}
    \centering

    \begin{tikzpicture}

        % \node at (-3,-3) {};

        % \node at (3,-3) {};

        \node at (-3,0) {\begin{tikzpicture}

	\draw[step=0.75cm,black,very thin] (0,0) grid (4.5,4.5);
	\filldraw[fill=black, draw=black, ultra thick] (0.375,0.375) circle (.315);
	\filldraw[fill=black, draw=black, ultra thick] (3.375,4.125) circle (.315);
	\filldraw[fill=black, draw=black, ultra thick] (1.125,1.875) circle (.315);
	\filldraw[fill=black, draw=black, ultra thick] (1.875,4.125) circle (.315);
	\filldraw[fill=black, draw=black, ultra thick] (1.875,2.625) circle (.315);
	\filldraw[fill=black, draw=black, ultra thick] (3.375,1.125) circle (.315);
	\filldraw[fill=black, draw=black, ultra thick] (4.125,0.375) circle (.315);

	\end{tikzpicture}};

        \node at (3,0){\begin{tikzpicture}

	\draw[step=0.75cm,black,very thin] (0,0) grid (4.5,4.5);
	\filldraw[fill=black, draw=black, ultra thick] (0.375,0.375) circle (.315);
	\filldraw[fill=black, draw=black, ultra thick] (3.375,4.125) circle (.315);
	\filldraw[fill=black, draw=black, ultra thick] (1.125,1.875) circle (.315);
	\filldraw[fill=black, draw=black, ultra thick] (1.875,4.125) circle (.315);
	\filldraw[fill=black, draw=black, ultra thick] (1.875,2.625) circle (.315);
	\filldraw[fill=black, draw=black, ultra thick] (3.375,1.125) circle (.315);
	\filldraw[fill=black, draw=black, ultra thick] (4.125,0.375) circle (.315);

%pieces

%4x1

	\fill[fill=gray,draw=black,very thin] (.75,0) rectangle (1.5,.75);
	\fill[fill=gray,draw=black,very thin] (1.5,0) rectangle (2.25,.75);
	\fill[fill=gray,draw=black,very thin] (2.25,0) rectangle (3,.75);
	\fill[fill=gray,draw=black,very thin] (3,0) rectangle (3.75,.75);

    \draw[black, ultra thick] (0.75,0) rectangle (3.75,0.75);

%2x2

	\fill[fill=green!65!black,draw=black,very thin] (0,3.75) rectangle (0.75,4.5);
	\fill[fill=green!65!black,draw=black,very thin] (0,3) rectangle (0.75,3.75);
	\fill[fill=green!65!black,draw=black,very thin] (.75,3.75) rectangle (1.5,4.5);
	\fill[fill=green!65!black,draw=black,very thin] (0.75,3) rectangle (1.5,3.75);

    \draw[black, ultra thick] (0, 3) rectangle (1.5, 4.5);

%tetromino T

	\fill[fill=yellow,draw=black,very thin] (1.5,3) rectangle (2.25,3.75);
	\fill[fill=yellow,draw=black,very thin] (2.25,3) rectangle (3,3.75);
	\fill[fill=yellow,draw=black,very thin] (2.25,3.75) rectangle (3,4.5);
	\fill[fill=yellow,draw=black,very thin] (2.25,2.25) rectangle (3,3);

%tetromino Z

	\fill[fill=red,draw=black,very thin] (3,2.25) rectangle (3.75,3);
	\fill[fill=red,draw=black,very thin] (3,3) rectangle (3.75,3.75);
	\fill[fill=red,draw=black,very thin] (3.75,3) rectangle (4.5,3.75);
	\fill[fill=red,draw=black,very thin] (3.75,3.75) rectangle (4.5,4.5);

%tetromino L

	\fill[fill=cyan,draw=black,very thin] (1.5,0.75) rectangle (2.25,1.5);
	\fill[fill=cyan,draw=black,very thin] (1.5,1.5) rectangle (2.25,2.25);
	\fill[fill=cyan,draw=black,very thin] (2.25,1.5) rectangle (3,2.25);
	\fill[fill=cyan,draw=black,very thin] (3,1.5) rectangle (3.75,2.25);

%3 I

	\fill[fill=orange,draw=black,very thin] (3.75,0.75) rectangle (4.5,1.5);
	\fill[fill=orange,draw=black,very thin] (3.75,1.5) rectangle (4.5,2.25);
	\fill[fill=orange,draw=black,very thin] (3.75,2.25) rectangle (4.5,3);

    \draw[black, ultra thick] (3.75,0.75) rectangle (4.5,3);

%3 L

	\fill[fill=deeppurple,draw=black,very thin] (0,1.5) rectangle (0.75,2.25);
	\fill[fill=deeppurple,draw=black,very thin] (0,2.25) rectangle (0.75,3);
	\fill[fill=deeppurple,draw=black,very thin] (0.75,2.25) rectangle (1.5,3);

%2

	\fill[fill=brown,draw=black,very thin] (0,0.75) rectangle (0.75,1.5);
	\fill[fill=brown,draw=black,very thin] (0.75,0.75) rectangle (1.5,1.5);

    \draw[black, ultra thick] (0,0.75) rectangle (1.5,1.5);

%1

	\fill[fill=blue!90!white,draw=black,very thin] (2.25,0.75) rectangle (3,1.5);

    \draw[black, ultra thick] (2.25,0.75) rectangle (3,1.5);

%Rest of the borders

    \draw[black, ultra thick] (0, 1.5) -- (0,3);
    \draw[black, ultra thick] (0,3) -- (1.5,3);
    \draw[black, ultra thick] (0,1.5) -- (0.75,1.5);
    \draw[black, ultra thick] (0.75, 1.5) -- (0.75, 2.25);
    \draw[black, ultra thick] (0.75, 2.25) -- (1.5, 2.25);
    \draw[black, ultra thick] (1.5, 2.25) -- (1.5, 3);

    \draw[black, ultra thick] (1.5, 0.75) -- (1.5, 2.25);
    \draw[black, ultra thick] (1.5, 2.25) -- (3.75, 2.25);
    \draw[black, ultra thick] (1.5, 0.75) -- (2.25, 0.75);
    \draw[black, ultra thick] (2.25, 1.5) -- (3.75, 1.5);
    \draw[black, ultra thick] (3.75,1.5) -- (3.75, 2.25);

    \draw[black, ultra thick] (2.25,2.25) -- (2.25, 3);
    \draw[black, ultra thick] (2.25, 3) -- (1.5,3);
    \draw[black, ultra thick] (1.5, 3) -- (1.5, 3.75);
    \draw[black, ultra thick] (1.5, 3.75) -- (2.25, 3.75);
    \draw[black, ultra thick] (2.25, 3.75) -- (2.25, 4.5);
    \draw[black, ultra thick] (2.25, 4.5) -- (3, 4.5);
    \draw[black, ultra thick] (3, 4.5) -- (3, 2.25);

    \draw[black, ultra thick] (3, 2.25) -- (3, 3.75);
    \draw[black, ultra thick] (3, 3.75) -- (3.75, 3.75);
    \draw[black, ultra thick] (3.75, 3.75) -- (3.75, 4.5);
    \draw[black, ultra thick] (3.75, 4.5) -- (4.5, 4.5);
    \draw[black, ultra thick] (4.5, 4.5) -- (4.5, 3);
    \draw[black, ultra thick] (4.5, 3) -- (3.75, 3);
    \draw[black, ultra thick] (3.75, 3) -- (3.75, 2.25);
	\end{tikzpicture}};
        
    \end{tikzpicture}
    
    \caption{A $6\times 6$ board with 7 blockers and a tiling solution.}
    \label{fig:enter-label}
\end{figure}
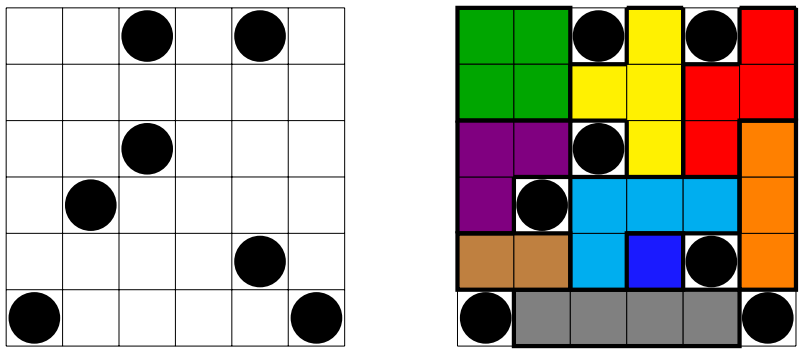

Let $m$ and $n$ be positive integers, and consider a rectangular grid of fixed orientation which has $m$ rows and $n$ columns of edge-connected square cells. For an integer $1\leq r\leq mn$, we consider the set $\B(m,n;r)$ of all such grids for which exactly $r$ of the square cells are \textit{blocked}. One could imagine placing a token on each of the $r$ cells, or coloring them with the same color while the remaining $mn-r$ squares are left blank. 

Our interest in the set $\B(m,n;r)$ arises from finite polyomino tiling puzzles. One can consider each grid in $\B(m,n;r)$ to be a puzzle board, and ask which of the boards in $\B(m,n;r)$ can be tiled with a prescribed set of free polyominoes. The problem of solving an individual finite polyomino tiling problem is $\mathcal{NP}$-complete in general \cite{ConLag,Lewis}, and it quickly becomes intractable to solve all boards in $\B(m,n;r)$ as the dimensions of the board and the number of distinct polyominoes grows. However, in the range where such computations are feasible, savings can be made by considering symmetry. If two boards in $\B(m,n;r)$ are equivalent via some symmetry, then any tiling solution for one board can be acted on by the same symmetry to solve the equivalent board. Thus, while there are ${mn\choose r}$ total boards in $\B(m,n;r)$, we only need to solve one representative from each equivalence class of boards under the group $G$ of symmetries acting on $\B(m,n;r)$. The number of equivalence classes can be computed precisely using Burnside's Lemma (as we do in \S 5), but a general lower bound for the number of inequivalent boards in $\B(m,n;r)$ is $\frac{1}{8}{mn\choose r}$ in the case of square boards, and $\frac{1}{4}{mn\choose r}$ in the case of non-square boards, each with $m,n>1$.

All boards in $\B(m,n;r)$ can be generated algorithmically via backtracking. As a brute force approach to producing only one board from each equivalence class, one could generate every board in $\B(m,n;r)$ and for each board, build all equivalent boards and then choose the minimal board from the equivalence class according to some lexicographic ordering. Here instead we take a combinatorial approach to classifying boards in $\B(m,n;r)$. Rather than exhaustively generating all of $\B(m,n;r)$, we produce only a subset of $\B(m,n;r)$ and we do not require any comparisons among stored boards. While our methods do not, in general, produce only one board from each equivalence class, we are able to avoid generating most of the redundant equivalent boards while guaranteeing that we retain at least one board from every equivalence class.   

We divide an $m\times n$ board into disjoint regions (depending on $m$ and $n$) and classify the individual boards in $\mathcal{B}(m,n;r)$ according to how many of their $r$ blocked cells are located in each region. We represent the counts of blocked cells in each region by an ordered tuple (or array) which we call a \emph{board partition}. We then define a subset $\overline{\B}(m,n;r)\subseteq \B(m,n;r)$ consisting of all those boards in $\B(m,n;r)$ whose board partition meets certain criteria, designed so that at least one board from each equivalence class appears in $\overline{\B}(m,n;r)$. Two boards with the same board partition may be equivalent under $G$, but this is the only circumstance in which redundancy occurs.  That is, if two boards in $\overline{\B}(m,n;r)$ are equivalent, then they must have the same board partition. Theorem \ref{thm:genthm} gives a weighted count of $\overline{\B}(m,n;r)$ which sums to $|\B(m,n;r)|={mn\choose r}$. Analogously, by Corollary \ref{cor:countformula} we can apply the same weights to the count of only those boards in $\overline{\B}(m,n;r)$ that are solvable for some particular polyomino tiling problem, to produce the total number of solvable boards in $\B(m,n;r)$. 

Our paper proceeds as follows. In Section \ref{sec:Prelim}, we review some basic definitions and results from group actions and polyomino tilings. Then in Section \ref{sec:main_thm}, we state and prove Theorem \ref{thm:genthm} and its corollary for solvable boards. Our restrictions on board partitions depend both on the parity of the board dimensions and on the symmetry group of the boards. Therefore in Section \ref{sec:BPcases}, we consider separately the cases of a square or non-square rectangular board with even or odd side lengths, and state our restriction criteria in each case. Formulas for the exact number of equivalence classes of $\B(m,n;r)$ under symmetry in each case are given in Section \ref{sec:Burnside}. We then share polyomino tiling results for a few representative tiling problems in Section \ref{sec:Results}. Finally, in Section \ref{sec:ReductionScale} we explain the circumstance in which redundant boards will occur in $\overline{\B}(m,n;r)$. We introduce a ratio which measures how close we are able to get to the optimal solution of a canonical board from each equivalence class, and give some general cases in which this optimal result is achieved. We also make two conjectures based on patterns we have observed in the values of this ratio.

\section{Preliminaries}\label{sec:Prelim}

Here we introduce some notation and briefly review some facts about group actions. We then describe a general class of polyomino tiling problems and discuss two approaches to solving a given tiling problem.

\subsection{Group Actions}

Let $G$ be a group acting on a set $S$. Then a fixed $g\in G$ determines a permutation of $S$, and we denote the image of $s\in S$ under $g$ as $g\cdot s$. Two elements $s,s'\in S$ are \textit{equivalent under $G$} if there is some $g\in G$ such that $g\cdot s=s'$. The \textit{equivalence class} (or \textit{orbit}) of $s$ under $G$ is the subset of $S$ given by $[s]_G:=\{g\cdot s:g\in G\}$. Let $\O_G(S)$ be the set of distinct equivalence classes of $S$ under $G$. The elements of $\O_G(S)$ form a partition of $S$. The \textit{stabilizer} of $s$ is the subgroup of $G$ consisting of all symmetries that fix $s$, that is,  $\mathrm{stab}_G(s):=\{g\in G:g\cdot s=s\}$. The set of elements of $S$ which are fixed by a particular $g\in G$ is denoted by $S^g:=\{s\in S:g\cdot s=s\}$. We will make use of the following two counting theorems.

\textbf{Orbit-Stabilizer Theorem:} For any $s\in S$, we have $|G|=|[s]_G|\cdot|\mathrm{stab}_G(s)|$.

\textbf{Burnside's Lemma:} The number of equivalence classes of $S$ under the action of $G$ is
\[|\O_G(S)|=\frac{1}{|G|}\sum_{g\in G}|S^g|.\]

The group of symmetries of a square $n\times n$ board (with $n>1$) is the dihedral group $D_4$, consisting of rotations about the center of the square by 0, 90, 180 and 270 degrees counter-clockwise (denoted by $R_0$, $R_{90}$, $R_{180}$, and $R_{270}$) and reflections across the horizontal, vertical, and diagonal lines bisecting the square ($H$, $V$, $D$ and $D'$, where $D$ is the reflection across the `main diagonal' running from the top left to bottom right of the square). The nontrivial subgroups of $D_4$ are the five groups of order two generated by $R_{180}$, $H$, $V$, $D$ and $D'$, the cyclic group of rotations $\cyc{R_{90}}=\{R_0,R_{90},R_{180},R_{270}\}$, and two noncyclic groups of order four: $\cyc{H,V}=\{R_0,H,V,R_{180}\}$ and  $\cyc{D,D'}=\{R_0,D,D',R_{180}\}$. The symmetry group of a non-square rectangular board with $m,n>1$ is given by $\cyc{H,V}$ and the symmetry group for a rectangular board with $m=1$ or $n=1$ (but not both) is given by $\cyc{R_{180}}$.

\subsection{Polyomino tiling}

 A \emph{polyomino} of order $n$ (also called an $n$-omino) is a shape made up of $n\geq 1$ edge-connected unit squares. We consider \emph{free} $n$-ominoes, which are equivalence classes of $n$-ominoes under rotations and reflections. In Table \ref{tab:PolyTable} we show all free polyominoes of order up to 4.

\begin{table}[ht]
\caption{Free polyominoes of order up to 4}
\medskip

\label{tab:PolyTable}
\centering

\begin{tikzpicture}

    \node at (-4,0.145) {\begin{tabular}{r | r | c}

	\textbf{Order} & \textbf{Name} & \textbf{Shape} \\

	\hline

	1 & & \\
	
	& Monomino & \begin{tikzpicture}[scale=0.35] 

	\draw[step = 1cm,black,thick] (0,0) grid (1,1);

	\end{tikzpicture}\\

	\hline

	2 & & \\

	& Domino & \begin{tikzpicture}[scale = 0.35]

	\draw[step=1cm,black,thick] (0,0) grid (1,2);

	\end{tikzpicture}\\

	\hline

	3 & & \\

	 & L-Tromino & \begin{tikzpicture}[scale=0.35]

	\draw[step=1cm,black,thick] (0,0) grid (2,2);
	\filldraw[fill=white,draw=white] (1.05,1.05) rectangle (2.05,2.05);

	\end{tikzpicture}\\

	 & I-Tromino  & \begin{tikzpicture}[scale=0.35]

	\draw[step=1cm,black,thick] (0,0) grid (1,3);

	\end{tikzpicture}\\

\end{tabular} };

    \node at (4,0) {

	\begin{tabular}{r | r | c}

	\textbf{Order} & \textbf{Name} & \textbf{Shape} \\

	\hline

	4 & & \\

	 & {L-Tetromino} & \begin{tikzpicture}[scale=0.35]

	\draw[step=1cm,black,thick] (0,0) grid (2,3);
	\filldraw[fill=white,draw=white] (1.05,1.05) rectangle (2.05,3.05);

	\end{tikzpicture}\\

	 & I-Tetromino & \begin{tikzpicture}[scale=0.35]

	\draw[step=1cm,black,thick] (0,0) grid (4,1);

	\end{tikzpicture}\\

	 & T-Tetromino & \begin{tikzpicture}[scale=0.35]

	\draw[step=1cm,black,thick] (0,0) grid (3,2);
	\filldraw[fill=white,draw=white] (-0.1,-0.1) rectangle (0.95,0.95);
	\filldraw[fill=white,draw=white] (2.05,-.1) rectangle (3.05,0.95);

	\end{tikzpicture}\\

	 & Z-Tetromino & \begin{tikzpicture}[scale=0.35]

	\draw[step=1cm,black,thick] (0,0) grid (2,3);
	\filldraw[fill=white,draw=white] (-0.1,-0.1) rectangle (0.95,0.95);
	\filldraw[fill=white,draw=white] (1.05,2.05) rectangle (2.05,3.05);

	\end{tikzpicture}\\

	 & Square-Tetromino & \begin{tikzpicture}[scale=0.35]

	\draw[step=1cm,black,thick] (0,0) grid (2,2);

	\end{tikzpicture}\\

\end{tabular} };

\end{tikzpicture}

\end{table}

 Problems related to polyominoes were first introduced by Golomb in his seminal 1965 book (updated in 1996, \cite{Gol}) and a thorough treatment may also be found in \cite{GS}. Much of the work on polyominoes centers on questions of tiling the lattice $\Z^2$ with copies of one or more polyominoes, or tiling a finite region of $\Z^2$ (such as our $m\times n$ blocked boards) with a prescribed set of polyominoes. A wealth of additional references on these topics are provided in \cite{GB}. A \textit{tiling} is a placement of the polyominoes onto the desired region so that the region is exactly covered with no gaps or overlapping polyominoes. Many board games and logic puzzles, including the commercially available games Genius Square, Blokus, Pentomino, and Quintillions, involve square or rectangular tiling problems together with elements of chance or strategy. 

We can describe a collection of finite polyomino tiling problems by specifying a set of boards $\B(m,n;r)$ and a set of free polyominoes whose orders sum to $mn-r$. We then ask which of the boards in $\B(m,n;r)$ can be tiled with the given set of polyominoes. An individual tiling problem can be solved using a backtracking algorithm, wherein one of the given polyominoes is placed in an available position, then the next, until either the problem is solved, or a polyomino can not be placed. In the latter case, the last placed polyomino is removed and placed in a different position. This process continues until the problem is solved or every placement is exhausted, at which point the problem is declared unsolvable. Recently, Garvie and Burkardt \cite{GB} proposed an alternative approach in which the tiling problem is turned into a linear system for which one seeks a binary solution. While this approach is not necessarily more efficient than backtracking, it has the advantage of being easily adaptable to a wide variety of finite regions and any set of repeated or unique polyominoes one wishes to use. The linear system resulting from a given tiling problem is underdetermined in general. In their paper, Garvie and Burkardt describe examples of systems of various sizes in $s$ variables. The medium to large problems ($200<s<70,000$) can be solved more efficiently using the CPLEX optimization package together with MATLAB. We use Garvie and Burkardt's computational methodology for our calculations in Section \ref{sec:Results}.

\section{The Main Theorem of Board Set Reduction}\label{sec:main_thm}

The following theorem gives general requirements for the reduced set of boards $\overline{\B}(m,n;r)$ and a formula by which we can recover the number of boards in $\mathcal{B}(m,n;r)$ via a weighted sum of the counts of boards with fixed board partitions in $\overline{\B}(m,n;r)$.

\begin{theorem}\label{thm:genthm}
    Let $G$ be the symmetry group acting on the set $\B(m,n;r)$ and suppose that a subset $\overline{\B}(m,n;r)\subseteq \B(m,n;r)$ satisfies the following conditions:
    \begin{enumerate}
        \item the set $\overline{\B}(m,n;r)$ is a disjoint union of sets $\pi_1,...,\pi_t$, where each $\pi_i$ is the set of all boards in $\B(m,n;r)$ with some fixed board partition.
        \medskip

         \item every board $B\in\B(m,n;r)$ is equivalent under $G$ to some $B'\in\overline{\B}(m,n;r)$.
        \medskip

        \item if $B,B'\in \overline{\B}(m,n;r)$ are equivalent under $G$, then they have the same board partition, and hence both lie in $\pi_i$ for some $1\leq i\leq t$.
    \end{enumerate}
    
    For each $1\leq i\leq t$, let $K_i\leq G$ be the subgroup consisting of all symmetries that preserve the set $\pi_i$. Then 
    \[|\B(m,n;r)|=\sum_{i=1}^t|\pi_i|\cdot [G:K_i].\]
\end{theorem}

\begin{proof}
For ease of notation, let $\B:=\B(m,n;r)$ and let $\overline{\B}:=\overline{\B}(m,n;r)$. For each $1\leq i\leq t$, the group $K_i$ acts on $\pi_i$, so let $\mathcal{O}_i:=\mathcal{O}_{K_i}(\pi_i)$. Then using the Orbit-Stabilizer Theorem and Lagrange's Theorem, we have
\begin{eqnarray}\label{sumcount}
\qquad\sum_{i=1}^t|\pi_i|[G:K_i]=\sum_{i=1}^t\left(\sum_{[B]_{K_i}\in\mathcal{O}_i}\frac{|K_i|}{|\mathrm{stab}_{K_i}(B)|}\right)\cdot\frac{|G|}{|K_i|}=\sum_{i=1}^t\sum_{[B]_{K_i}\in\mathcal{O}_i}\frac{|G|}{|\mathrm{stab}_{K_i}(B)|}.\end{eqnarray}

For each $B\in\pi_i$, we have $\mathrm{stab}_{K_i}(B)\leq\mathrm{stab}_G(B)$, since $K_i\leq G$. If a symmetry $g\in G$ fixes $B$ then it must at least preserve the board partition of any board in $\pi_i$. Hence $g\in K_i$, and so $\mathrm{stab}_{K_i}(B)=\mathrm{stab}_G(B)$. Then from (\ref{sumcount}), we have

\begin{eqnarray}\label{orbits}\sum_{i=1}^t|\pi_i|[G:K_i]=\sum_{i=1}^t\sum_{[B]_{K_i}\in\mathcal{O}_i}\frac{|G|}{|\mathrm{stab}_{G}(B)|}=\sum_{i=1}^t\sum_{[B]_{K_i}\in\mathcal{O}_i}|[B]_G|.\end{eqnarray}

Next, if $B$ and $B'$ in $\overline{\B}$ are equivalent under $K_i$ then they are equivalent under $G$, so each $[B]_{K_i}\subseteq [B]_G\cap \overline{\B}$. But by condition (3), $[B]_G\cap\overline{\B}$ is contained in $\pi_i$. Then if $B,B'\in\pi_i$ are equivalent under $G$, they must be equivalent under $K_i$, so $[B]_{K_i}=[B]_G\cap \overline{\B}$. Now the double sum in (\ref{orbits}) can be read as a single sum over the restriction of some equivalence classes $[B]_G$ to $\overline{\B}$. But by condition (2), every board in $\B$ is equivalent to some board in $\overline{\B}$, so the sum is over $[B]_G\cap \overline{\B}$ for \textit{all} $[B]_G\in\O_G(\B)$ and it is equivalent to simply sum over all $[B]_G\in\O_G(\B)$, since the summand does not depend on the representative board $B$ from the equivalence class.

Therefore from (\ref{orbits}), we have
\[\sum_{i=1}^t|\pi_i|[G:K_i]=\sum_{[B]_G\in\mathcal{O}_G(\mathcal{B})}|[B]_G|=|\B(m,n;r)|.\]

\end{proof}

In Section \ref{sec:BPcases} we will show that for each case of board dimensions, the subsets $\overline{\B}(m,n;r)$ we define satisfy conditions (1), (2) and (3) from Theorem \ref{thm:genthm}. The counting formula in Theorem \ref{thm:genthm} can be applied as well to the subset of solvable boards for some tiling problem.

\begin{corollary}\label{cor:countformula}
    Let $\overline{\B}(m,n;r)\subseteq \B(m,n;r)$ satisfy the three conditions of Theorem \ref{thm:genthm}, and for each $1\leq i\leq t$, let $S_i\subseteq\pi_i$ be the subset of boards in $\pi_i$ for which some prescribed tiling problem is solvable. Then the total number of boards in $\B(m,n;r)$ which are solvable for the given tiling problem is
    \[\sum_{i=1}^t|S_i|\cdot [G:K_i].\]    
    \end{corollary}

    \begin{proof} A board $B\in \B(m,n;r)$ is solvable if and only if every board equivalent to $B$ under $G$ is solvable. Therefore the proof proceeds as in Theorem \ref{thm:genthm} after restricting both $\B$ and $\overline{\B}$ to just the solvable boards in each set.\end{proof}

    \begin{rem} Analogously, one could use Corollary \ref{cor:countformula} to count the \textit{unsolvable} boards in $\B(m,n;r)$. Indeed, this corollary holds generally whenever the sets $S_i$ consist of all boards in $\pi_i$ which meet some condition that is preserved under symmetry.\end{rem}

\section{Board Partition Restrictions}\label{sec:BPcases}

The way in which we reduce the set of boards $\B(m,n;r)$ will depend on both the parity of the dimensions $m\times n$ and the symmetry group $G$ acting on $\B(m,n;r)$. In each case below we specify a certain board type, divide that board into regions, and describe the form of the board partition for that board type. The nonnegative integer components of the board partition give the number of blockers in each region of the board, as denoted in an accompanying figure. We then set restrictions on the board partitions to produce subsets $\overline{\B}(m,n;r)$ which satisfy the conditions of Theorem \ref{thm:genthm}. The proofs in this section will involve acting on boards by a sequence of symmetries. It is important to note that although the action of a symmetry on a board will permute its board partition, when we use the notation $\lambda_i$, $\delta_i$, or $c$ (components of the board partitions), we will always be referring to the count of blocked cells in the region labeled by that symbol in the figure. For example, $\lambda_1$ will be the generic number of blocked cells in the upper left corner region of any board under discussion.

\subsection{Square boards with even side length}

Let $k$ be a positive integer and consider the collection $\mathcal{B}(2k,2k;r)$. We can divide any board in this collection into four {$k\times k$} quadrants as in Figure \ref{fig:BoardQuads}. 

\begin{figure}[h]
    \centering
    \caption{A square board divided into four equal regions.}
    \label{fig:BoardQuads}

    \begin{tikzpicture}[scale = 1]

	   \draw[step=1cm,black,very thin] (0,0) grid (2,2);
	   \node[black] at (0.5,0.5) {$\lambda_4$};
	   \node[black] at (1.5,0.5) {$\lambda_3$};
	   \node[black] at (0.5,1.5) {$\lambda_1$};
	   \node[black] at (1.5,1.5) {$\lambda_2$};

    \end{tikzpicture}
\end{figure}

The board partition for any board in $\mathcal{B}(2k,2k;r)$ is given by $(\lambda_1,\lambda_2,\lambda_3,\lambda_4)$, where each quadrant contains some $\lambda_i$ blocked cells, numbered as in Figure \ref{fig:BoardQuads}. Necessarily, we have $0\leq \lambda_i\leq k^2$ for each $1\leq i\leq 4$, and $\sum\lambda_i=r$.

The following theorem gives conditions for producing a set of board partitions which represent the entire collection $\mathcal{B}(2k,2k;r)$ via symmetry.

\begin{theorem} \label{thm:sqeven}
Let $\overline{\B}(2k,2k;r)$ be the subset of $\B(2k,2k;r)$ consisting of all boards whose board partition $(\lambda_1,\lambda_2,\lambda_3,\lambda_4)$  has the following properties:
\begin{enumerate}[(i)]
    \item $\lambda_1\geq \lambda_i$ for all $i>1$,
    \item $\lambda_2\geq\lambda_4$,
    \item if $\lambda_1=\lambda_2$ then $\lambda_3\geq\lambda_4$.
\end{enumerate}
Then $\overline{\B}(2k,2k;r)$ meets the three conditions of Theorem \ref{thm:genthm}.
\end{theorem}

\begin{proof}
    By construction, $\overline{\B}(2k,2k;r)$ contains all boards in $\B(2k,2k;r)$ having certain board partitions, so condition (1) is met. Recall that the symmetry group for $\B(2k,2k;r)$ is $G=D_4$. Let $B\in\B(2k,2k;r)$. First we can apply one of the four rotations so that the resulting board satisfies (i). Next, we apply $D$ if necessary so that the resulting board satisfies both (i) and (ii). Finally, if we now have $\la_1=\la_2$, we apply $V$ if necessary to satisfy (iii). The result of this sequence of symmetries applied to $B$ is a board $B'\in \overline{\B}(2k,2k;r)$ which is equivalent to $B$ under $D_4$. Therefore condition (2) of Theorem \ref{thm:genthm} is met.
    
    Next we show that condition (3) is also satisfied. Let $B\in\overline{\B}(2k,2k;r)$ have board partition $(\lambda_1,\lambda_2,\lambda_3,\lambda_4)$, which satisfies properties (i) through (iii). We must show that each symmetry of $D_4$ applied to $B$ produces either another board with the same board partition, or a board whose partition fails at least one of the three properties, and hence is outside of $\overline{\B}(2k,2k;r)$. Clearly the identity of $D_4$ preserves the board partition of $B$. Now if $\lambda_1>\lambda_i$ for all $i>1$, then any nontrivial symmetry except for $D$ will violate (i), and $D$ will violate (ii) if $\la_2>\la_4$ or preserve the board partition if $\la_2=\la_4$. 
    
    If $\lambda_1=\lambda_i$ for some $i>1$, then we have some cases to consider. If $\lambda_1=\lambda_2=\lambda_3=\lambda_4$, then all symmetries preserve the board partition of $B$. If $\lambda_1$ is equal to two other $\lambda_i$, then we must have $\lambda_1=\lambda_2=\lambda_3>\lambda_4$. In this case, $D'$ preserves the board partition, and every nontrivial symmetry except for $D'$ moves the blocked cells from the lower left quadrant to a different quadrant, which will violate one of the properties. If $\lambda_1$ equals exactly one other $\lambda_i$, then we must have either $\lambda_1=\lambda_2>\lambda_3\geq \lambda_4$ or $\lambda_1=\lambda_3>\lambda_2\geq \lambda_4$. In the first case, if $\la_3=\la_4$ then $V$ preserves the board partition, and otherwise all nontrivial symmetries violate properties (i) or (ii). In the second case, only $D'$ preserves the board partition unless, in addition, $\la_2=\la_4$, in which case both $D$ and $R_{180}$ also preserve the board partition. All other nontrivial symmetries violate property (i).
\end{proof}

To make the implementation of Corollary \ref{cor:countformula} straightforward, we catalog the subgroups $K$ which preserve each board partition in $\overline{\B}(2k,2k;r)$.

\begin{corollary} \label{cor:sqeven_weights} Let $K\leq D_4$ be the subgroup of symmetries that preserves the set of boards with board partition $(\lambda_1,\lambda_2,\lambda_3,\lambda_4)$. The table below gives $K$ and the index $[D_4:K]$ for all board partitions represented in $\overline{\B}(2k,2k;r)$.

 \[  \begin{array}{|l|c|c|}
    \hline
    \mathrm{Board\, partition\, type}&K&[D_4:K]\\
    \hline
    \lambda_1=\lambda_2=\lambda_3=\lambda_4&D_4&1\\
    \lambda_1=\lambda_3>\lambda_2=\lambda_4&\cyc{D,D'}&2\\
    \lambda_1=\lambda_2>\lambda_3=\lambda_4&\cyc{V}&4\\
    \lambda_1=\lambda_3,\lambda_2\ne\lambda_4&\cyc{D'}&4\\
    \lambda_2=\lambda_4,\lambda_1\ne\lambda_3&\cyc{D}&4\\
    \mathrm{All\, others}&\cyc{e}&8\\
    \hline
    \end{array}\]
    
\end{corollary}

\begin{proof} The table above follows from the verification of condition (3) in the proof of Theorem \ref{thm:sqeven}. More directly, each of the 10 subgroups of $D_4$ (including $D_4$ itself) determines necessary conditions under which a board partition $(\lambda_1,\lambda_2,\lambda_3,\lambda_4)$ is preserved. The four subgroups not listed in the table are not maximal with respect to preserving particular board partitions.
\end{proof}

\subsection{Square boards with odd side length}

Next we consider the collection $\mathcal{B}(2k+1,2k+1;r)$ for $k\geq 1$. We divide the board into nine regions: four corners of size $k\times k$, four strips between each corner of size $k\times 1$ (or $1\times k$), and the center of size $1\times 1$ as can be seen in Figure \ref{fig:OddSquareBoard}. 

\begin{figure}[h]
    \centering
    \caption{A square board, of odd side length, divided into nine regions.}
    \label{fig:OddSquareBoard}

    \begin{tikzpicture}[scale=0.75]

        \draw[step=1cm,black,very thin] (0,0) rectangle (2,5);
        \draw[step=1cm,black,very thin] (2,0) rectangle (3,5);
        \draw[step=1cm,black,very thin] (3,0) rectangle (5,5);
        \draw[step=1cm,black,very thin] (0,2) rectangle (5,3);

        \node[black] at (1,4) {$\lambda_1$};
        \node[black] at (2.5,4) {$\delta_1$};
        \node[black] at (4,4) {$\lambda_2$};
        \node[black] at (4,2.5) {$\delta_2$};
        \node[black] at (4,1) {$\lambda_3$};
        \node[black] at (2.5,1) {$\delta_3$};
        \node[black] at (1,1) {$\lambda_4$};
        \node[black] at (1,2.5) {$\delta_4$};
        \node[black] at (2.5,2.5) {$c$};

    \end{tikzpicture}
    
\end{figure}

Let 
$\displaystyle \begin{pmatrix}
    \la_1 & \la_2 & \la_3 & \la_4 \\
    \de_1 & \de_2 & \de_3 & \de_4 \\
\end{pmatrix}_c$
be the board partition which gives the number of blocked cells in each region. We must have $0\leq \la_i\leq k^2$ and $0\leq \de_i \leq k$  for each $1\leq i \leq 4$, with $c\in\{0,1\}$, and $\sum \la_i + \sum \de_i + c = r$. The following Theorem describes the reduced set $\overline{\B}(2k+1,2k+1;r)$.

\begin{theorem} \label{thm:sqodd}

    Let $\overline{\B}(2k+1,2k+1;r)$ be the subset of $\B(2k+1,2k+1;r)$ consisting of all boards whose board partition 
    $\displaystyle \begin{pmatrix}
    \la_1 & \la_2 & \la_3 & \la_4 \\
    \de_1 & \de_2 & \de_3 & \de_4 \\
    \end{pmatrix}_c$
    satisfies:

\begin{multicols}{2}
\begin{enumerate}[(i)]
\item $\la_1 \geq \la_i$ for all $i > 1$,
\item $\la_2\geq \la_4$,
\item if $\la_1 = \la_2$ then $\la_3\geq\la_4$,
\item if $\la_1=\la_3$  and $\la_2\ne\la_4$ then 
    \begin{itemize}
        \item $\de_1\geq\de_2$, and
        \item if $\de_1=\de_2$ then $\de_3\geq\de_4$,
    \end{itemize}
\item if $\la_2=\la_4$ and $\la_1\ne\la_3$ then 
    \begin{itemize}
        \item $\de_1\geq\de_4$, and
        \item if $\de_1=\de_4$ then $\de_2\geq\de_3$,
    \end{itemize}
\item if $\la_1=\la_2>\la_3=\la_4$ then $\de_2\geq\de_4$,
\item if $\la_1=\la_3>\la_2=\la_4$ then 
    \begin{itemize}
        \item $\de_1\geq\de_i$
        \item if $\de_1=\de_2$ then $\de_3\geq\de_4$
        \item if $\de_1=\de_3$ then $\de_2\geq\de_4$
        \item if $\de_1=\de_4$ then $\de_2\geq\de_3$.
    \end{itemize}
\item if  $\la_1=\la_2=\la_3=\la_4$ then
    \begin{itemize}
        \item $\de_1\geq\de_i$
        \item $\de_2\geq\de_4$
        \item if $\de_1=\de_2$ then $\de_3\geq\de_4$,
    \end{itemize}
\end{enumerate}
\end{multicols}
    
Then $\overline{\B}(2k+1,2k+1;r)$ meets the three conditions of Theorem \ref{thm:genthm}
\end{theorem}

\begin{proof}
Condition (1) of Theorem \ref{thm:genthm} is satisfied by construction. To prove that condition (2) is satisfied, let $B\in\B(2k+1,2k+1;r)$. We will apply a sequence of symmetries to $B$ in order to arrive at an equivalent board in $\overline{\B}(2k+1,2k+1;r)$. 

If $\la_1=\la_2=\la_3=\la_4$ for $B$, then we can apply some rotation $R$ followed by $V$ then $D'$ as necessary to find a board equivalent to $B$ which satisfies (viii). Otherwise, the $\lambda_i$ for $B$ are not all equal. We first apply the sequence of symmetries $R$ (some rotation), $D$, $V$ as necessary to guarantee (i), (ii) and (iii) are all satisfied, as in the proof of Theorem \ref{thm:sqeven}.

    At this point there are four remaining cases to consider.

    \textbf{Case 1}: If $\la_1=\la_3$ and $\la_2\ne\la_4$, then we apply $D'$ if necessary for (iv), which does not disrupt properties (i), (ii) or (iii).

    \textbf{Case 2}: If $\la_2=\la_4$ and $\la_1\ne\la_3$, then we apply $D$ if necessary for (v), which does not disrupt properties (i), (ii) or (iii). Note that if we need to use $D$ here, then $D$ will not have already been performed to achieve (ii) since $\la_2=\la_4$.

   \textbf{Case 3}: If $\la_1=\la_2>\la_3=\la_4$, then we apply $V$ if necessary to achieve (vi), which does not disrupt properties (i), (ii) or (iii). Note that if we need to use $V$ here, then $V$ will not have already been performed to achieve (iii) since $\la_3=\la_4$.

    \textbf{Case 4}: If $\la_1=\la_3>\la_2=\la_4$ then we will need to ensure that (vii) is satisfied. We first perform one of $D$, $D'$ or $R_{180}$ to guarantee $\de_1\geq \de_i$ for all $i$. Each of these three symmetries will preserve properties (i), (ii) and (iii). If $\de_1>\de_i$ for all $i>1$ then we are done. Otherwise, if $\de_1$ is equal to exactly one other $\de_i$, we do $D'$ if $\de_1=\de_2$ and $\de_3<\de_4$, or $R_{180}$ if $\de_1=\de_3$ and $\de_2<\de_4$, or $D$ if $\de_1=\de_4$ and $\de_2<\de_3$. If $\de_1=\de_2=\de_3$ then since $\de_1\geq\de_4$, (vii) is satisfied. If $\de_1=\de_2=\de_4>\de_3$, then $D'$ will ensure (vii). Finally, if $\de_1=\de_3=\de_4>\de_2$ then $R_{180}$ will ensure (vii).

Finally, we show that condition (3) of Theorem \ref{thm:genthm} is met. Let $B\in\overline{\B}(2k+1,2k+1;r)$. Then its board partition satisfies properties (i) through (viii). We will show that every nontrivial symmetry of $D_4$ either preserves the board partition or takes $B$ outside of $\overline{\B}(2k+1,2k+1;r)$.

First, if $\lambda_1>\lambda_i$ for all $i>1$, then all nontrivial symmetries except for $D$ violate (i), and $D$ violates either (ii) or (v) unless $\la_2=\la_4$, $\de_1=\de_4$, and $\de_2=\de_3$, in which case $D$ preserves the board partition.

If $\lambda_1=\lambda_i$ for some $i>1$, then our cases for the $\la_i$ are as in the proof of Theorem \ref{thm:sqeven} but there are additional considerations for the $\de_i$. If $\lambda_1$ is equal to two other $\lambda_i$, then we must have $\lambda_1=\lambda_2=\lambda_3>\lambda_4$. In this case, all nontrivial symmetries but $D'$ violate (i), (ii) or (iii), and $D'$ violates (iv) unless $\de_1=\de_2$ and $\de_3=\de_4$, whence $D'$ preserves the board partition. If $\lambda_1$ equals exactly one other $\lambda_i$, then we must have either $\lambda_1=\lambda_2>\lambda_3\geq \lambda_4$ or $\lambda_1=\lambda_3>\lambda_2\geq \lambda_4$. In the first case, if $\la_3=\la_4$ and $\de_2=\de_4$ then $V$ preserves the board partition, and otherwise all nontrivial symmetries violate one of one of (i), (ii) or (vi). In the second case, if $\la_2\ne\la_4$ but $\de_1=\de_2$ and $\de_3=\de_4$, then $D'$ preserves the board partition and all other nontrivial symmetries violate one of (i), (ii) or (iii). If $\la_2=\la_4$, and all $\de_i$ are equal, then $\cyc{D,D'}$ preserves the board partition and the other non-identity symmetries violate (i). If $\la_2=\la_4$ and $\de_1=\de_3\ne\de_2=\de_4$, then $R_{180}$ preserves the board partition but all other non-identity symmetries violate either (i) or (viii).

Finally, if all four $\la_i$ are equal, then we consider the $\de_i$ similarly to how we considered the $\la_i$ in the cases above. The circumstances under which the board partition is preserved are cataloged in the following table, otherwise each nontrivial symmetry violates at least one condition of membership in $\overline{\B}(2k+1,2k+1;r)$.

  \[ \begin{array}{|l|c|}
    \hline
    \mathrm{Conditions\, on\, the\,} \de_i&\mathrm{Symmetries}\\
    \hline
    \de_1=\de_2=\de_3=\de_4&D_4\\
    \de_1=\de_2>\de_3=\de_4&\cyc{D'}\\
    \de_1=\de_3>\de_2=\de_4&\cyc{H,V}\\
    \de_1=\de_3\geq\de_2>\de_4&\cyc{H}\\
    \de_1>\de_3,\de_2=\de_4&\cyc{V}\\
    \hline\end{array}\]
\end{proof}

Here we catalog the subgroups which preserve each board partition in $\overline{\B}(2k+1,2k+1;r)$. 

\begin{corollary} \label{cor:sqodd_weights} Let $K\leq D_4$ be the subgroup of symmetries that preserve the set of boards with board partition $\displaystyle \begin{pmatrix}
    \la_1 & \la_2 & \la_3 & \la_4 \\
    \de_1 & \de_2 & \de_3 & \de_4 \\
    \end{pmatrix}_c$. The table below gives $K$ and the index $[D_4:K]$ for all board partitions represented in $\overline{\B}(2k+1,2k+1;r)$

  \[ \begin{array}{|l|c|c|}
    \hline
    \mathrm{Board\, partition\, type}&K&[D_4:K]\\
    \hline
\la_1=\la_2=\la_3=\la_4, \de_1=\de_2=\de_3=\de_4&D_4&1\\
\la_1=\la_2=\la_3=\la_4, \de_1=\de_2>\de_3=\de_4&\cyc{D'}&4\\
\la_1=\la_2=\la_3=\la_4, \de_1=\de_3>\de_2=\de_4&\cyc{H,V}&2\\
\la_1=\la_2=\la_3=\la_4, \de_1=\de_3\geq\de_2>\de_4&\cyc{H}&4\\
\la_1=\la_2=\la_3=\la_4, \de_1>\de_3,\de_2=\de_4&\cyc{V}&4\\
\hline
    \lambda_1=\lambda_3>\lambda_2=\lambda_4,\de_1=\de_2=\de_3=\de_4&\cyc{D,D'}&2\\
     \lambda_1=\lambda_3>\lambda_2=\lambda_4,\de_1=\de_3, \de_2=\de_ 4,\de_2\ne\de_3&\cyc{R_{180}}&4\\
    \lambda_1=\lambda_2>\lambda_3=\lambda_4,\de_2=\de_4&\cyc{V}&4\\
    \lambda_1=\lambda_3,\la_2\ne\la_4,\de_1=\de_2,\de_3=\de_4&\cyc{D'}&4\\
    \lambda_2=\lambda_4,\la_1\ne\la_3,\de_1=\de_4,\de_2=\de_3&\cyc{D}&4\\
    \mathrm{All\, others}&\cyc{e}&8\\
    \hline
    \end{array}\]

\end{corollary}

\begin{proof}
Proceed as in Corollary \ref{cor:sqeven_weights}.
\end{proof}

\subsection{Non-square boards with even side lengths}

Let $k,l\geq 1$ with $k\not=l$ and consider $\mathcal{B}(2k,2l;r)$. We divide the board into four $k\times l$ quadrants similarly to the square boards with even side lengths, as seen in Figure \ref{fig:EvenNonSquareBoard}.

\begin{figure}[h]
    \centering
    \caption{A non-square board, of even side lengths, divided into four regions.}
    \label{fig:EvenNonSquareBoard}

    \begin{tikzpicture}[scale=1]

        \draw[step=1cm,black,very thin] (0,0) rectangle (4,1);
        \draw[step=1cm,black,very thin] (2,0) rectangle (4,2);
        \draw[step=1cm,black,very thin] (0,0) rectangle (2,2);
        
        \node[black] at (1,1.5) {$\lambda_1$};
        \node[black] at (3,1.5) {$\lambda_2$};        \node[black] at (3,0.5) {$\lambda_3$};        \node[black] at (1,0.5) {$\lambda_4$};        
        
    \end{tikzpicture}
    
\end{figure}

The board partition for any board in $\mathcal{B}(2k,2l;r)$ is of the form $(\la_1,\la_2,\la_3,\la_4)$, and we require $0\leq\la_i\leq kl$ for each $1\leq i\leq 4$, and $\sum\la_i=r$. 

\begin{theorem} \label{thm:recteven}

Let $\overline{\mathcal{B}}(2k,2l;r)$ be the subset of $\mathcal{B}(2k,2l;r)$ consisting of boards whose board partition $(\la_1,\la_2,\la_3,\la_4)$ satisfies:

\begin{enumerate}[(i)]
    
    \item $\la_1\geq\la_i$ for all $i>1$,

    \item if $\la_1=\la_2$ then $\la_3\geq\la_4$,

    \item if $\la_1=\la_3$ then $\la_2\geq\la_4$,

    \item if $\la_1=\la_4$ then $\la_2\geq\la_3$.
\end{enumerate}
Then $\overline{\mathcal{B}}(2k,2l;r)$ satisfies the three conditions of Theorem \ref{thm:genthm}.
\end{theorem}

\begin{proof}

    Recall that the group of symmetries acting on a $2k\times 2l$ board is $G=\cyc{H,V}$ and note that condition (1) is satisfied by construction. Now, let $B\in\mathcal{B}(2k,2l;r)$. To show that $B$ is equivalent to some $B'\in \overline{\mathcal{B}}(2k,2l;r)$,  we first apply a symmetry of $G$ to ensure that $\la_1\geq\la_i$. Now if $\la_1>\la_i$ for all $i>1$ we are done. Otherwise, if $\la_1$ is equal to exactly one other $\la_i$, we may apply $V$, $R_{180}$ or $H$ to achieve (ii), (iii) or (iv), respectively, none of which disrupt (i). Now suppose $\la_1$ is equal to exactly two other $\la_i$. If $\la_1\ne\la_4$ then properties (ii) and (iii) are guaranteed by (i). If $\la_1\ne\la_3$ then apply $V$. If $\la_1\ne\la_2$ then apply $R_{180}$. Finally, if all four $\la_i$ are equal, then (ii) through (iv) are all met.  Therefore, some sequence of symmetries acting on $B$ produces an equivalent board in $\overline{\mathcal{B}}(2k,2l;r)$, so we have shown that condition (2) is satisfied.

    Next, we show that $\overline{\B}(2k,2l;r)$ satisfies condition (3). To this end, we observe that any symmetry applied to a board $B\in \overline{\mathcal{B}}(2k,2l;r)$ produces either a board with the same board partition or one which fails at least one of the four properties in the statement of this theorem. Applying $R_0$ preserves all board partitions.

    If $B$ has $\la_1>\la_i$ for all $i>1$, then every non-identity symmetry will produce a board which violates (i). Next, if $\la_1$ equals exactly one other $\la_i$, we have three scenarios. If $\la_1=\la_2$ then both $H$ and $R_{180}$ will violate (i) and $V$ will violate (iii) unless $\la_3=\la_4$, in which case $V$ preserves the board partition. If $\la_1=\la_3$ then $H$ and $V$ will both violate (i) and $R_{180}$ will either violate (ii) if $\la_2>\la_4$ or preserve the board partition if $\la_2=\la_4$. Finally, if $\la_1=\la_4$, then both $V$ and $R_{180}$ will violate (i) and $H$ will violate (iv) unless $\la_2=\la_3$, when it will preserve the board partition.

    The only way for $\la_1$ to equal exactly two other $\la_i$ is if $\la_1=\la_2=\la_3>\la_4$. In this case applying $H$ violates (i), applying $R_{180}$ violates (ii), and applying $V$ violates (iii). If all $\la_i$ are equal, then all of $G$ preserves the board partition. Thus, we satisfy condition (3).
    
\end{proof}

We summarize the conditions under which board partitions are fixed in the following corollary.

\begin{corollary}\label{cor:recteven}
    Let $\pi$ be the set of all boards in $\overline{\mathcal{B}}(2k,2l;r)$ with fixed board partition $(\la_1,\la_2,\la_3,\la_4)$, and let $K\leq \cyc{H,V}$ be the subgroup of symmetries that preserve $\pi$. The table below gives $K$ and the index $[\cyc{H,V} : K]$ for all board partitions represented in $\overline{\mathcal{B}}(2k,2l;r)$.

    \[ \begin{array}{|l|c|c|}
    \hline
    \mathrm{Board\, partition\, type}&K&[\cyc{H,V}:K]\\
    \hline
    \la_1=\la_2=\la_3=\la_4&\cyc{H,V}&1\\
       \la_1=\la_2>\la_3=\la_4&\cyc{V}&2\\
    \la_1=\la_3>\la_2=\la_4&\cyc{R_{180}}&2\\
    \la_1=\la_4>\la_2=\la_3&\cyc{H}&2\\
        
    \mathrm{All\, others}&\cyc{e}&4\\
    \hline
    \end{array}\]
\end{corollary}

\begin{proof}
    Proceed as in Corollary \ref{cor:sqeven_weights}.
\end{proof}

\subsection{Non-square boards with odd side lengths}

Let $k$ and $l$ be positive integers with $k\ne l$ and consider the collection $\mathcal{B}(2k+1,2l+1;r)$. We can divide the board into nine regions: four corners of dimension $k\times l$, two $k\times 1$ vertical strips, two $1\times l$ horizontal strips, and a $1\times 1$ center, as seen in Figure \ref{fig:OddNonSquareBoard}. 

\begin{figure}[h]
    \centering
    \caption{A non-square board, of odd side lengths, divided into nine regions.}
    \label{fig:OddNonSquareBoard}

    \begin{tikzpicture}[scale=1]

        \draw[step=1cm,black,very thin] (0,0) rectangle (2.2,2.5);
        \draw[step=1cm,black,very thin] (2.2,0) rectangle (2.8,2.5);
        \draw[step=1cm,black,very thin] (2.8,0) rectangle (5,2.5);
        \draw[step=1cm,black,very thin] (0,1) rectangle (5,1.5);

        \node[black] at (1,2) {$\lambda_1$};
        \node[black] at (2.5,2) {$\delta_1$};
        \node[black] at (4,2) {$\lambda_2$};
        \node[black] at (4,1.25) {$\delta_2$};
        \node[black] at (4,0.5) {$\lambda_3$};
        \node[black] at (2.5,0.5) {$\delta_3$};
        \node[black] at (1,0.5) {$\lambda_4$};
        \node[black] at (1,1.25) {$\delta_4$};
        \node[black] at (2.5,1.25) {$c$};        
        
    \end{tikzpicture}
    
\end{figure}

Again, we let $\displaystyle \begin{pmatrix}
    \la_1 & \la_2 & \la_3 & \la_4 \\
    \de_1 & \de_2 & \de_3 & \de_4 \\
\end{pmatrix}_c$ be the board partition denoting the number of blocked cells in each section. We require that $0\leq \la_i\leq kl$ for $1\leq i\leq 4$, $0\leq\de_i\leq k$ for $i\in\{1,3\}$, and $0\leq\de_i\leq l$ for $i\in\{2,4\}$, as well as $c\in\{0,1\}$ and $\sum \la_i +\sum\de_i+c=r$. 

\begin{theorem} \label{thm:rectodd}

Let $\overline{\mathcal{B}}(2k+1,2l+1;r)$ be the subset of $\mathcal{B}(2k+1,2l+1;r)$ consisting of boards whose board partition $\begin{pmatrix}
    \la_1 & \la_2 & \la_3 & \la_4 \\
    \de_1 & \de_2 & \de_3 & \de_4 \\
\end{pmatrix}_c$ satisfies: 

\begin{enumerate}[(i)]
    \item $\la_1\geq\la_i$ for all $i>1$,

    \item if $\la_1=\la_3$ then $\la_2\geq\la_4$; if in addition $\la_2=\la_4$, then $\de_1\geq\de_3$; further, if $\de_1=\de_3$, then $\de_2\geq\de_4$,

    \item if $\la_1=\la_2$ then $\la_3\geq\la_4$; if in addition $\la_3=\la_4$, then $\de_2\geq\de_4$,

    \item if $\la_1=\la_4$ then $\la_2\geq\la_3$; if in addition $\la_2=\la_3$, then $\de_1\geq\de_3$.

\end{enumerate}

Then, $\overline{\mathcal{B}}(2k+1,2l+1;r)$ satisfies the conditions of Theorem \ref{thm:genthm}.
    
\end{theorem}

\begin{proof}

    Condition (1) is satisfied by construction. Let $B\in\mathcal{B}(2k+1,2l+1;r)$. We can apply some element of $G=\cyc{H,V}$ to ensure that $\la_1\geq\la_i$ for any $i$. If $\la_1>\la_i$ for each $i$, we are done. If not, and $\la_1$ is equal to exactly one other $\la_i$, we can then apply $R_{180}$, $V$, or $H$ as necessary to satisfy (ii), (iii), or (iv) respectively. None of these will contradict (i). Next, suppose $\la_1$ is equal to exactly two other $\la_i$. If $\la_1\not=\la_4$, then (ii), (iii), and (iv) are guaranteed to be satisfied. On the other hand if $\la_1\not=\la_2$, then apply $R_{180}$ and if $\la_1\not=\la_3$, then apply $V$. Last, if all $\la_i$ are equal and $\de_1<\de_3$ apply $H$. If $\de_2<\de_4$, apply $V$. So, there is a sequence of symmetries whose application to $B$ produces a board in $\overline{\mathcal{B}}(2k+1,2l+1;r)$ and thus condition (2) is satisfied.

    Next, we show that condition (3) is satisfied by demonstrating that any symmetry applied to a board $B\in\overline{\mathcal{\B}}(2k+1,2l+1;r)$ either produces a board with the same board partition or a board whose board partition violates at least one of properties (i) through (iv). Applying $R_0$ preserves all board partitions. If $\la_1>\la_i$ for each $i>1$ then any non-identity symmetry applied to $B$ will violate (i). 
    
    We now consider when $\la_1$ equals exactly one other $\la_i$. First, consider $\la_1=\la_2$. Both $H$ and $R_{180}$ violate (i). If $\la_3>\la_4$, or if $\la_3=\la_4$ and $\de_2>\de_4$, then $V$ violates (iii), otherwise $V$ preserves the board partition.
    
    Next, if $\la_1=\la_3$, then both $H$ and $V$ violate (i). If $\la_2>\la_4$, or if $\la_2=\la_4$ but either $\de_1>\de_3$ or $\de_2>\de_4$, then $R_{180}$ violates (ii), otherwise $R_{180}$ preserves the board partition.

    If $\la_1=\la_4$, then $V$ and $R_{180}$ violate (i). If $\la_2>\la_3$, or if $\la_2=\la_3$ and $\de_1>\de_3$, then $H$ violates (iv), and otherwise $H$ preserves the board partition.

    If $\la_1$ equals exactly two other $\la_i$, then just as in Theorem \ref{thm:recteven}, each non-identity symmetry violates one of the conditions of $\overline{\B}(2k+1,2l+1,r)$. 
    
    Last, we consider when all four $\la_i$ are equal. We must have $\de_1\geq\de_3$ and $\de_2\geq\de_4$. If $\de_1>\de_3$ while $\de_2=\de_4$, then $H$ and $R_{180}$ violate (ii) and (iv) while $V$ preserves the board partition. If $\de_1=\de_3$ while $\de_2>\de_4$, then $V$ and $R_{180}$ violate (ii) and (iii) while $H$ preserves the board partition. If both are strict inequalities, then all non-identity symmetries violate at least one of the conditions of $\overline{\B}(2k+1,2l+1,r)$, otherwise if they are both equalities then the board partition is preserved by all symmetries.

\end{proof}

\begin{corollary}\label{cor:rectodd}
    Let $\pi$ be the set of all boards in $\overline{\mathcal{B}}(2k+1,2l+1;r)$ with fixed board partition $\begin{pmatrix}
    \la_1 & \la_2 & \la_3 & \la_4 \\
    \de_1 & \de_2 & \de_3 & \de_4 \\
\end{pmatrix}_c$, and let $K\leq \cyc{H,V}$ be the subgroup of symmetries that preserve $\pi$. The table below gives $K$ and the index $[\cyc{H,V} : K ]$ for all board partitions represented in $\overline{\mathcal{B}}(2k+1,2l+1;r)$.

    \[ \begin{array}{|l|c|c|}
    \hline
    \mathrm{Board\, partition\, type}&K&[\cyc{H,V}:K]\\
    \hline
    \la_1=\la_2=\la_3=\la_4,\de_1=\de_3,\de_2=\de_4&\cyc{H,V}&1\\
    \la_1=\la_3>\la_2=\la_4,\de_1=\de_3,\de_2=\de_4&\cyc{R_{180}}&2\\
    \la_1=\la_2,\la_3=\la_4,\de_2=\de_4, \mathrm{if}\, \la_2=\la_3\, \mathrm{then}\,\de_1\ne\de_3&\cyc{V}&2\\
    \la_1=\la_4,\la_2=\la_3,\de_1=\de_3,\,\mathrm{if}\,\la_2=\la_3\,\mathrm{then}\,\de_2\ne\de_4&\cyc{H}&2\\
        
    \mathrm{All\, others}&\cyc{e}&4\\
    \hline
    \end{array}\]
    
\end{corollary}

\begin{proof}
    Proceed as in Corollary \ref{cor:sqeven_weights}.
\end{proof}

\subsection{Non-square boards with odd and even side lengths}

Let $k$ and $l$ be positive integers. Without loss of generality we consider the set $\mathcal{B}(2k,2l+1;r)$. We can divide this boar into six regions, including four corner regions of size $(k\times l)$ and two middle regions of size $(k\times 1)$, as seen in Figure \ref{fig:OddEvenBoard}.

\begin{figure}[h]
    \centering
    \caption{A non-square board with even and odd side lengths, divided into six regions.}
    \label{fig:OddEvenBoard}

    \begin{tikzpicture}[scale=1]

        \draw[step=1cm,black,very thin] (0,0) rectangle (5,2);
        \draw[step=1cm,black,very thin] (0,1) rectangle (5,2);
        \draw[step=1cm,black,very thin] (2,0) rectangle (3,2);

        \node[black] at (1,1.5) {$\lambda_1$};
        \node[black] at (2.5,1.5) {$\delta_1$};
        \node[black] at (4,1.5) {$\lambda_2$};
        \node[black] at (4,0.5) {$\lambda_3$};
        \node[black] at (2.5,0.5) {$\delta_2$};
        \node[black] at (1,0.5) {$\lambda_4$};
        
    \end{tikzpicture}
    
\end{figure}

Here, we let $(\la_1,\la_2,\la_3,\la_4,\de_1,\de_2)$ be the board partition. We require that $0\leq \la_i\leq kl$ for each $1\leq i \leq 4$, also that $0\leq \de_i\leq k$ for $1\leq i \leq 2$, and lastly that $\sum \la_i +\sum \de_i=r$. 

\begin{theorem} \label{thm:rectboth}

Let $\overline{\mathcal{B}}(2k,2l+1;r)$ be the subset of $\mathcal{B}(2k,2l+1;r)$ consisting of boards whose board partition $(\la_1,\la_2,\la_3,\la_4,\de_1,\de_2)$ satisfies: 

    \begin{enumerate}[(i)]
        \item $\la_1\geq\la_i$ for all $i>1$,

        \item if $\la_1=\la_2$ then $\la_3\geq\la_4$,
        
        \item if $\la_1=\la_3$ then $\la_2\geq\la_4$; if in addition $\la_2=\la_4$, then $\de_1\geq\de_2$,

        \item if $\la_1=\la_4$ then $\la_2\geq\la_3$; additionally, if $\la_2=\la_3$, then $\de_1\geq\de_2$.

    \end{enumerate}

    Then, $\overline{\mathcal{B}}(2k,2l+1;r)$ satisfies the conditions of Theorem \ref{thm:genthm}.
    
\end{theorem}

\begin{proof}

    Condition (1) of Theorem \ref{thm:genthm} is satisfied by construction. Consider a board $B\in\mathcal{B}(2k,2l+1;r)$. We satisfy (i) by applying one of the symmetries of $G=\cyc{H,V}$ to enforce $\la_1\geq\la_i$. Now, if $\la_1>\la_i$ for all $i>1$, we are done. Otherwise, if $\la_1$ is equal to exactly one other $\la_i$, we can apply $V$, $R_{180}$, or $H$ as needed to satisfy (ii), (iii), and (iv), respectively, without contradicting (i). Consider now if $\la_1$ equals exactly two other $\la_i$. If $\la_1\neq\la_4$, then (i) guarantees that (ii) and (iii) are satisfied. However, if $\la_1\neq\la_2$, apply $R_{180}$ and if $\la_1\neq\la_3$ apply $V$. Again, this will satisfy (ii), (iii) and (iv) without interfering with (i). Last, if all $\la_i$ are equal and $\de_1<\de_2$, then apply $H$, otherwise, we are done. So, we can apply a sequence of symmetries to $B$ to produce a board in $\overline{\mathcal{B}}(2k,2l+1;r)$, satisfying condition (2).

    Now let $B\in\overline{\mathcal{B}}(2k,2l+1;r)$. As always, if $\la_1>\la_i$ then any non-identity symmetry will violate (i). Next, consider when $\la_1$ is equal to exactly one other $\la_i$. If $\la_1=\la_2$, then applying $H$ or $R_{180}$ will violate (i). If also $\la_3>\la_4$ then $V$ violates (iii), otherwise $V$ preserves the board partition.

    If $\la_1=\la_3$, then applying $H$ or $V$ will violate (i). If $\la_2>\la_4$ or if $\la_2=\la_4$ and $\de_1>\de_2$, then $R_{180}$ violates (ii), otherwise it preserves the board partition.

    If $\la_1=\la_4$, then applying $V$ or $R_{180}$ will violate (i). If $\la_2>\la_3$, or if $\la_2=\la_3$ and $\de_1>\de_2$, then $H$ violates (iv), otherwise $H$ preserves board partition.

    Next, we consider when $\la_1$ is equal to exactly two other $\la_i$. This is only possible when $\la_1=\la_2=\la_3>\la_4$. Applying $H$ violates (i), applying $R_{180}$ violates (ii), and applying $V$ violates (iii). Last we consider when all $\la_i$ are equal to each other. Applying $V$ preserves the board partition and applying $H$ or $R_{180}$ will preserve board partition when $\de_1=\de_2$. If $\de_1>\de_2$, then both $H$ and $R_{180}$ will violate (ii) and (iv). And so, condition (3) is satisfied.

\end{proof}

\begin{corollary}\label{cor:rectevenodd}

Let $\pi$ be the set of all boards in $\overline{\mathcal{B}}(2k,2l+1;r)$ with fixed board partition $(\la_1,\la_2,\la_3,\la_4,\de_1,\de_2)$, and let $K\leq\cyc{H,V}$ be the subgroup of symmetries that preserve $\pi$. The table below gives $K$ and the index of $[\cyc{H,V} : K]$ for all board partitions represented in $\overline{\mathcal{B}}(2k,2l+1;r)$.

 \[ \begin{array}{|l|c|c|}
    \hline
    \mathrm{Board\, partition\, type}&K&[\cyc{H,V}:K]\\
    \hline
    \la_1=\la_2=\la_3=\la_4,\de_1=\de_2&\cyc{H,V}&1\\
    \la_1=\la_2=\la_3=\la_4,\de_1\ne\de_2&\cyc{V}&2\\  
    \la_1=\la_2>\la_3=\la_4&\cyc{V}&2\\ 
    \la_1=\la_4>\la_2=\la_3,\de_1=\de_2&\cyc{H}&2\\
        
    \mathrm{All\, others}&\cyc{e}&4\\
    \hline
    \end{array}\]
\end{corollary}

\begin{proof}
    Proceed as in Corollary \ref{cor:sqeven_weights}.
\end{proof}

\subsection{Boards with one side of length one.} For completion, we include the case where a board has only one row or column. Without loss of generality, consider $\mathcal{B}(1,2l;r)$ for some integer $l\geq 1$. The only difference when we consider $\mathcal{B}(1,2l+1;r)$ is that there is a center cell that is fixed by any symmetry. Each board in this set is divided into 2 equal $1\times l$ pieces as seen in Figure \ref{fig:OneSideBoard}.

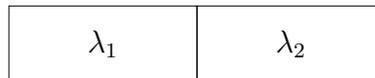
\begin{figure}[h]
    \centering
    \caption{A board with a side length of one divided into two regions.}
    \label{fig:OneSideBoard}

    \medskip

    \begin{tikzpicture}[scale=1]

        \draw[step=1cm,black,very thin] (0,0) rectangle (2.5,1);
        \draw[step=1cm,black,very thin] (2.5,0) rectangle (5,1);
        
        \node[black] at (1.25,0.5) {$\lambda_1$};
        \node[black] at (3.75,0.5) {$\lambda_2$};
        
    \end{tikzpicture}
    
\end{figure}

The board partition for any board in $\mathcal{B}(1,2l;r)$, is $(\la_1,\la_2)$. For this board, we require that $0\leq\la_i\leq l$ for $1\leq i\leq 2$ and $\la_1+\la_2=r$. Note that if the board is $1\times (2l+1)$, then we have the board partition $(\la_1,\la_2)_c$. We require for the odd side length case that $0\leq\la_i\leq l$, $c\in\{0,1\}$, and $\la_1+\la_2+c=r$. 

\begin{theorem}\label{thm:1byn}

    Let $\overline{\mathcal{B}}(1,2l;r)$ be the subset of $\mathcal{B}(1,2l;r)$ consisting of boards whose board partition $(\la_1,\la_2)$ satisfies $\la_1\geq\la_2$. Then $\overline{\mathcal{B}}(1,2l;r)$ satisfies the three conditions of Theorem \ref{thm:genthm}.
    
\end{theorem}

\begin{proof}

    Condition (1) is satisfied by construction. Now, let $B\in\mathcal{B}(1,2l;r)$. Recall the symmetry group for $1\times 2l$ boards is $G=\cyc{R_{180}}$. If $\la_1<\la_2$ then we apply $R_{180}$ so that the resulting board satisfies (i). Thus, we satisfy condition (2). Next, let $B\in\overline{\mathcal{B}}(1,2l;r)$. Applying $R_0$ keeps the board in the same board partition. If $\la_1>\la_2$, then applying $R_{180}$ will violate (i), but if $\la_1=\la_2$, it will preserve the board partition. Thus, condition (3) is satisfied.
    
\end{proof}

\begin{corollary}\label{cor:1byn}

    Let $\pi$ be the set of all boards in $\overline{\mathcal{B}}(1,2l;r)$ with fixed board partition $(\la_1,\la_2)$, and let $K\leq\cyc{R_{180}}$ be the subgroup of symmetries that preserve $\pi$. The following table gives $K$ and the index of $[\cyc{R_{180}} : K]$ for all board partitions represented in $\overline{\mathcal{B}}(1,2l;r)$.

     \[ \begin{array}{|l|c|c|}
    \hline
    \mathrm{Board\, partition\, type}&K&[\cyc{R_{180}}:K]\\
    \hline
    \la_1=\la_2&\cyc{R_{180}}&1\\
    \mathrm{All\, others}&\cyc{e}&2\\
    \hline
    \end{array}\]
    
\end{corollary}

\begin{proof}

    Proceed as in Corollary \ref{cor:sqeven_weights}.
    
\end{proof}

\section{Counting Equivalence Classes}\label{sec:Burnside}

It is straightforward to count the equivalence classes of $\B(m,n;r)$ under the action of $G$ using Burnside's Lemma. We record formulas for each case here. We provide proofs of only the first two results; the remaining cases are quite similar. Throughout, we assume $k,l\geq 1$.

\begin{proposition}[Square boards with even side length] \label{prop:Burn_squareeven} Let $\B=\B(2k,2k;r)$ and $G=D_4$. Then
\[\left|\O_G(\B)\right|=\frac{1}{8}\left({4k^2\choose r}+2{k^2\choose \frac{r}{4}}+3{2k^2\choose \frac{r}{2}}+2\sum_{t=0}^{r}{2k\choose t}{k(2k-1)\choose \frac{(r-t)}{2}}\right).\]
\end{proposition}

\begin{proof} All boards in $\B$ are fixed by $R_0$. If $4\mid r$ then some boards in $\B$ will be fixed by both $R_{90}$ and $R_{270}$. These are counted by freely assigning $\frac{r}{4}$ blockers to one quadrant, for which there are ${k^2\choose \frac{r}{4}}$ choices, after which the placement of the remaining blockers in the other three quadrants is determined by symmetry. Similarly, if $2\mid r$ then some boards in $\B$ will be fixed by either $H$, $V$ or $R_{180}$. The boards in each category are counted in the same way: by freely assigning $\frac{r}{2}$ of the blockers to one half of the board, which determines the placement of the remaining blockers in the other half, yielding ${2k^2\choose \frac{r}{2}}$ boards in each fixed set. Finally, the boards fixed by $D$ or $D'$ will be symmetric across one or the other diagonal. The fixed set for each diagonal symmetry is counted by freely placing some $0\leq t\leq r$ of the $r$ blockers on the diagonal line of cells, and freely placing half of the remaining blockers on one side of the diagonal, with the remaining blocker positions determined on the other side of the diagonal.

\end{proof}

\begin{proposition} [Square boards with odd side length] \label{prop:Burn_squareodd} Let $\B=\B(2k+1,2k+1;r)$ and $G=D_4$. Then
\begin{multline*}\left|\O_G(\B)\right|=\frac{1}{8}\left({(2k+1)^2\choose r}+2{k(k+1)\choose \frac{r}{4}}+2{k(k+1)\choose \frac{r-1}{4}}+ {2k(k+1)\choose\lfloor\frac{r}{2}\rfloor}\right.\\ 
\left.+4\sum_{t=0}^{r}{2k+1\choose t}{k(2k+1)\choose \frac{r-t}{2}}\right).\end{multline*}
\end{proposition}

\begin{proof}
    Our counts of each fixed set are similar to those in the even square case, but now we separately consider possible blockers in the middle row or column. As before, all boards in $\B$ are fixed by $R_0$. If $4\mid r$ then some boards are fixed by $R_{90}$ and $R_{270}$. These boards are formed by assigning $\frac{r}{4}$ blockers to a region containing one $k\times k$ quadrant and one adjacent $k\times 1$ (or $1\times k$) strip. However, if $4\mid (r-1)$, we can form fixed boards by assigning one blocker to the central $1\times 1$ square, then freely assigning $\frac{r-1}{4}$ blockers to a quadrant and adjacent strip. For these two cases, the remaining quadrants and strips are determined by symmetry. Similarly, some boards are fixed by $R_{180}$. If $2\mid r$, then there is no blocker in the central square and if $2\nmid r$ there is a blocker placed in the central square. Then $\lfloor\frac{r}{2}\rfloor$ blockers are freely placed in an $L$-shaped region containing two adjacent quadrants, the strip between them, and one other strip adjacent to either quadrant. The remaining blockers are determined by these placements.

    To count the boards fixed by $H$ (or $V$), we start by placing some $0\leq t\leq r$ of the $r$ blockers freely on the horizontal (or vertical) midline. Then we freely place half of the remaining blockers on one half of the board, these halves being split by the midline. The rest of the blocker positions are determined by the first half. The fixed sets for $D$ and $D'$ are counted analogously to Proposition \ref{prop:Burn_squareeven}.
\end{proof}

\begin{proposition}[Non-square boards with even side lengths]\label{prop:Burn_nonsquareEven} For $k\ne l$, let $\B=\B(2k,2l;r)$ and $G=\cyc{H,V}$. Then
\[\left|\O_G(\B)\right|=\frac{1}{4}\left({4kl\choose r}+3{2kl\choose \frac{r}{2}}\right).\]
\end{proposition}

\begin{proposition}[Non-square boards with odd side lengths]\label{prop:Burn_nonsquareOdd} For $k\ne l$, let $\B:=\B(2k+1,2l+1;r)$ and $G=\cyc{H,V}$. Then 
\begin{multline*}\left|\O_G(\B)\right|=\frac{1}{4}\left({(2k+1)(2l+1)\choose r}+{2kl+k+l\choose \lfloor\frac{r}{2}\rfloor}\right.\\ 
\left.+\sum_{t=0}^{r}{2k+1\choose t}{2kl+l\choose\frac{r-t}{2}}+\sum_{t=0}^{r}{2l+1\choose t}{2kl+k\choose\frac{r-t}{2}}\right)\end{multline*}
\end{proposition}

\begin{proposition}[Non-square boards with odd and even side lengths]\label{prop:Burn_rectOddEven} Let $\B:=\B(2k,2l+1;r)$ and $G=\cyc{H,V}$. Then
\[\left|\O_G(\B)\right|=\frac{1}{4}\left({2k(2l+1)\choose r}+2{k(2l+1)\choose\frac{r}{2}}+
\sum_{t=0}^{r}{2k\choose t}{2kl\choose\frac{r-t}{2}}\right).\]

\end{proposition}

\begin{proposition}[Boards with one side of length one] For $n>1$, let $\B:=\B(1,n;r)$ and $G=\cyc{R_{180}}$. 
\begin{enumerate}[(i)]
\item If $n=2k$ then
\[\left|\O_G(\B)\right|=\frac{1}{2}\left({2k\choose r}+{k\choose \frac{r}{2}}\right).\]
    \medskip
    \item If $n=2k+1$ then 
\[\left|\O_G(\B)\right|=\frac{1}{2}\left({2k+1\choose r}+ {k\choose\lfloor\frac{r}{2}\rfloor}\right).\]
\end{enumerate}
\end{proposition}

\section{Collections of Tiling Problems}\label{sec:Results}

Here we apply our board set reduction to solve several classes of polyomino tiling problems. The code we used to solve these problems implements the method of Garvie and Burkardt \cite{GB} and was written using Matlab, Python, and CPLEX. All of the code is available in GitHub \cite{Jgit}. 

A collection of tiling problems consists of a fixed set $\B(m,n;r)$ of boards and a given set of free polyominoes (not necessarily distinct). We first generate all board partitions which satisfy the conditions specified in the relevant theorem of Section \ref{sec:BPcases}, and create each board having one of the allowed board partitions. For each board, we build an associated linear system of equations which depends upon both the board and the given set of polyominoes. Each system is solved or declared unsolvable. We then make a weighted count of the unsolvable boards using the appropriate corollary from Section \ref{sec:BPcases} together with Corollary \ref{cor:countformula} to recover the total number of boards in $\B(m,n;r)$ which can not be tiled with the given set of polyominoes.

\begin{example}[Genius Square]

Genius Square is a commercially available board game consisting of a $6\times 6$ game board and 7 blockers. One copy of each of the nine free polyominoes of order one through four must be used to tile the board (see Figure \ref{GSB}). The game includes a set of dice which one rolls to determine the placement of the blockers; every dice roll leads to a solvable board, but only 62,208 boards can be reached via the dice. Here we use our methods to try to solve all $\sim 8.3$ million boards. We apply Theorem \ref{thm:sqeven} to the set $\B(6,6;7)$ to determine 20 allowable board partitions. In Table \ref{tab:GSParts} we list these board partitions along with the subgroup $K$ under which the board partition is invariant and the index $[D_4:K]$. 

The reduced board set $\overline{\B}(6,6;7)$ contains 1,521,054 boards, which is approximately 18.22\% of the boards in $\B(6,6;7)$. Exactly 29,813 of the boards in $\overline{\B}(6,6;7)$ are unsolvable. Using  Corollaries \ref{cor:countformula} and \ref{cor:sqeven_weights}, we find that of the ${36\choose 7}\approx 8.3$ million boards in $\B(6,6;7)$, only 172,440 boards, amounting to only around 2\% of the total boards, are unsolvable. 

From Proposition \ref{prop:Burn_squareeven}, we find that the number of equivalence classes of boards in $\B(6,6;7)$ is 1,044,690. Solving just one board from each equivalence class would require solving about 12.51\% of the total boards compared to the 18.22\% our method produced. So, we see that we ended up solving 1.46 times as many boards as we needed to. 

 In an unpublished project \cite{JSeniorProj}, the first author used an elementary form of the techniques developed in this paper to solve this class of tiling problems related to Genius Square. Jensen further developed an alternative set of dice which guarantees 93,750 solvable boards.

\begin{table}[h]

\caption{Board partitions of a $6\times 6$ board with 7 blockers}

\label{tab:GSParts}

\centering
\begin{tabular}{| c | c | c | }
\hline
    Board Partition & $K$ & $[D_4:K]$  \\
    \hline
    (7,0,0,0) & $\cyc{D}$ & 4 \\
    (6,1,0,0) & $\cyc{e}$ & 8 \\
    (6,0,1,0) & $\cyc{D}$ & 4 \\
    (5,2,0,0) & $\cyc{e}$ & 8 \\
    (5,0,2,0) & $\cyc{D}$ & 4 \\
    (5,1,1,0) & $\cyc{e}$ & 8 \\
    (5,1,0,1) & $\cyc{D}$ & 4 \\
    (4,3,0,0) & $\cyc{e}$ & 8 \\
    (4,0,3,0) & $\cyc{D}$ & 4 \\
    (4,2,1,0) & $\cyc{e}$ & 8 \\
    (4,2,0,1) & $\cyc{e}$ & 8 \\
    (4,0,2,1) & $\cyc{e}$ & 8 \\
    (4,1,1,1) & $\cyc{D}$ & 4 \\
    (3,3,1,0) & $\cyc{e}$ & 8 \\
    (3,1,3,0) & $\cyc{D'}$ & 4 \\
    (3,2,2,0) & $\cyc{e}$ & 8 \\
    (3,2,0,2) & $\cyc{D}$ & 4 \\
    (3,2,1,1) & $\cyc{e}$ & 8 \\
    (3,1,2,1) & $\cyc{D}$ & 4 \\
    (2,2,2,1) & $\cyc{D'}$ & 4 \\  
    \hline
\end{tabular}
\end{table}
\end{example}

\begin{example}[Pentominoes on a chessboard]

Next, we consider tiling an $8\times 8$ board with 4 blockers using one copy of each of the twelve free polyominoes of order 5. Working in $\B(8,8;4)$, we apply Theorem \ref{thm:sqeven} which produces 8 board partitions. These can be seen in Table \ref{tab:Pentomino}.

With this set of board partitions we generate 175,516 boards, which is about 27.6\% of the total number of boards. After attempting to solve these boards with CPLEX, we find that 1,900 of them are unsolvable. Via Corollaries \ref{cor:countformula} and \ref{cor:sqeven_weights} we calculate that 9,552 boards out the total ${64\choose 4}$, around 1.5\%, are unsolvable.

By Proposition \ref{prop:Burn_squareeven}, there are 79,920 equivalence classes in $\B(8,8;4)$, so one board from each class would amount to 12.57\% of the total number of boards. This means we have computed about 2.2 times the number of boards needed. 

\begin{table}[h]
\caption{Board partitions of an $8\times 8$ board with 4 blockers}

\label{tab:Pentomino}
\centering
\begin{tabular}{| c | c | c | }
\hline
    Board Partition & $K$ & $[D_4:K]$  \\
    \hline
    (1,1,1,1) & $\cyc{D_4}$ & 1 \\
    (2,0,2,0) & $\cyc{D,D'}$ & 2 \\
    (2,1,0,1) & $\cyc{D}$ & 4 \\
    (2,1,1,0) & $\cyc{e}$ & 8 \\
    (2,2,0,0) & $\cyc{V}$ & 4 \\
    (3,0,1,0) & $\cyc{D}$ & 4 \\
    (3,1,0,0) & $\cyc{e}$ & 8 \\
    (4,0,0,0) & $\cyc{D}$ & 4 \\
    \hline
\end{tabular}
\end{table}
\end{example}

\begin{example}[L-trominoes on a $5\times 7$ board]
% There are 68252 unsolvable boards -> 247694 after symmetries
Here, we consider a $5\times 7$ game board with 5 blockers, tiled with ten copies of a single free polyomino, the L-shaped order 3 polyomino. We apply Theorem \ref{thm:rectodd} to $\B(5,7;5)$ which produces 276 board partitions. A table containing these partitions is available on GitHub \cite{Jgit}.

With these 276 board partitions, we generate 89,278 boards in $\overline{\B}(5,7;5)$, which is about 27.5\% of the total ${35\choose 5}$ boards. After attempting to solve these boards, we find that 68,252 of them are unsolvable. Corollaries \ref{cor:countformula} and \ref{cor:rectodd} result in 247,694 unsolvable boards in $\B(5,7;5)$, which is around 76.3\% of the total boards.

By Proposition \ref{prop:Burn_nonsquareOdd}, the minimum number of boards to solve (one from each equivalence class) is 81,648, about 25.15\% of the total number of boards. So our methods required us to solve only about 1.09 times as many boards as necessary, which is the closest to optimal of our examples so far.
\end{example}

\begin{example}[I-polyominoes on a $5\times 5$ board]

\end{example}
% There are 459,652 unsolvable boards -> 3,213,292 after symmetries
Next, we tile a $5\times 5$ board with 10 blockers using a monomino, a domino, and the I-tromino, I-tetromino and I-pentomino. Using Theorem \ref{thm:sqodd}, we determine 1,221 board partitions which generate $467,376$ boards in $\overline{\B}(5,5;10)$, approximately $14.30\%$ of the total ${25 \choose 10}$ boards. A table containing these board partitions is provided in \cite{Jgit}. Within $\overline{\B}(5,5;10)$,  we find 459,652 unsolvable boards, and Corollaries \ref{cor:countformula} and \ref{cor:sqodd_weights} give us a total of 3,213,292 unsolvable boards in $\B(5,5;10)$, which is approximately 98.3\% of the total boards. %\st{input number of unsolvable boards.}

Now, using Proposition \ref{prop:Burn_squareodd}, we find there are $410,170$ equivalence classes of boards in $\B(5,5;10)$. So we attempted to solve about 1.14 times as many boards as needed.

\begin{example}[One type of pentomino on a $6\times 8$ board]\label{ex:pent6by8}
% There are 2572 unsolvable boards -> 10288 after symmetries
\end{example}

Now, we consider attempting to tile a $6\times 8$ board with 3 blockers using only P-pentominoes. With Theorem \ref{thm:recteven}, we determine 5 board partitions that generate 4,324 of boards in $\overline{\B}(6,8;3)$, which is exactly 25\% of the total ${48\choose 3}$ boards. A table containing these board partitions is provided in \cite{Jgit}. With these board partitions, we find 2,572 unsolvable boards, then using Corollaries \ref{cor:countformula} and \ref{cor:recteven}, we find a total of 10,288 unsolvable boards. This accounts for approximately 59.5\% of the total boards.

With Proposition \ref{prop:Burn_nonsquareEven} we find that there are 4,324 equivalence classes of boards in $\mathcal{B}(6,8;3)$. Notably this is exactly the number of boards we generated in $\overline{\B}(6,8;3)$. In fact, this is always the case when we have an even-sided rectangular board with an odd number of blockers. We give this general result in Proposition \ref{prop:even_rect_reduction}.

\begin{example}[Tetrominoes on a $6\times 7$ board]
% There are 925,208 unsolvable boards -> 3,137,062 after symmetries

\end{example}

Last, we consider tiling a $6\times7$ board with 6 blockers using 2 of all of the tetrominoes except for the S-tetromino which was used only once. We are now working in $\B(6,7;6)$ and applying Theorem \ref{thm:rectboth}, we find 115 board partitions. These 115 board partitions generate 1,556,344 boards, which is approximately 29.67\% of the total ${42\choose 6}$ boards. A table containing these board partitions can be found in \cite{Jgit}. Using these board partitions, we find 925,208 unsolvable boards. Then, when we use Corollaries \ref{cor:countformula} and \ref{cor:rectevenodd}, we find a total of 3,137,062 unsolvable boards. These unsolvable boards account for approximately 59.8\% of all of the possible boards.

Making use of Proposition \ref{prop:Burn_rectOddEven}, we find that there are 1,312,957 equivalence classes of boards. Thus, we attempt to solve 1.19 times as many boards as we would need to if we only solved one from each equivalence class.

\section{Canonical Representation} \label{sec:ReductionScale}

In the set $\overline{\B}(m,n;r)$ as defined in each case of \S4, two boards can be equivalent only if they share the same board partition. More precisely, they must share a board partition which assigns to two or more regions of the board the same number of blocked cells.

\begin{example}
   The following two boards in $\overline{\B}(6,6;7)$ each have board partition $(3,1,2,1)$ and are equivalent via the diagonal symmetry $D$.

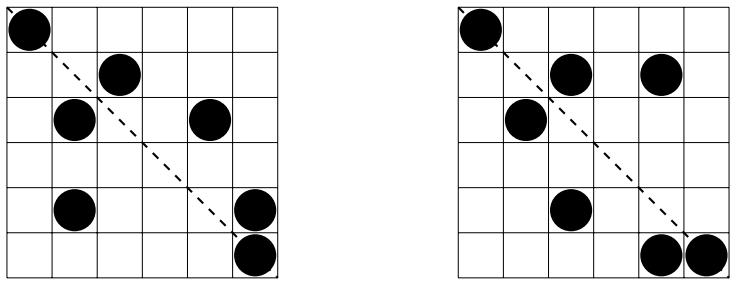
    \begin{center}
        \begin{tikzpicture}
    
            \node at (-3,0) {\begin{tikzpicture}[scale=0.6]
    
                \draw[step=1cm,black,very thin] (0,0) grid (6,6);
    
                \filldraw[fill=black,draw=black,ultra thick] (5.5,0.5) circle (0.42);
                \filldraw[fill=black,draw=black,ultra thick] (5.5,1.5) circle (0.42);
                \filldraw[fill=black,draw=black,ultra thick] (4.5,3.5) circle (0.42);
                \filldraw[fill=black,draw=black,ultra thick] (2.5,4.5) circle (0.42);
                \filldraw[fill=black,draw=black,ultra thick] (1.5,3.5) circle (0.42);
                \filldraw[fill=black,draw=black,ultra thick] (1.5,1.5) circle (0.42);
                \filldraw[fill=black,draw=black,ultra thick] (0.5,5.5) circle (0.42);

                \draw[dashed, black, thick] (0,6) -- (6,0);
    
            \end{tikzpicture}};
    
            \node at (3,0) {\begin{tikzpicture}[scale=0.6]
    
                \draw[step=1cm,black,very thin] (0,0) grid (6,6);
    
                \filldraw[fill=black,draw=black,ultra thick] (5.5,0.5) circle (0.42);
                \filldraw[fill=black,draw=black,ultra thick] (4.5,0.5) circle (0.42);
                \filldraw[fill=black,draw=black,ultra thick] (1.5,3.5) circle (0.42);
                \filldraw[fill=black,draw=black,ultra thick] (0.5,5.5) circle (0.42);
                \filldraw[fill=black,draw=black,ultra thick] (4.5,4.5) circle (0.42);
                \filldraw[fill=black,draw=black,ultra thick] (2.5,1.5) circle (0.42);
                \filldraw[fill=black,draw=black,ultra thick] (2.5,4.5) circle (0.42);

                \draw[dashed, black, thick] (0,6) -- (6,0);
      
            \end{tikzpicture}};
            
        \end{tikzpicture}

    \end{center}

    In contrast, any board in $\overline{\B}(6,6;7)$ with board partition $(4,2,1,0)$ will be inequivalent to all other boards in $\overline{\B}(6,6;7)$ since any nontrivial symmetry of $D_4$ would produce a board whose board partition violates the requirements of Theorem \ref{thm:sqeven}.
\end{example}

The board partitions in which equivalent boards can occur are specified in the tables in Corollaries \ref{cor:sqeven_weights}, \ref{cor:sqodd_weights}, \ref{cor:recteven}, \ref{cor:rectodd}, \ref{cor:rectevenodd} and \ref{cor:1byn}. Since the majority of these board partitions require an even number of blocked cells, we get closer to a canonical representation of boards when $r$ is odd. In Propositions \ref{prop:even_rect_reduction}, \ref{prop:onerow} and \ref{prop:oneblocker} below, we record three general cases in which $r$ is odd and $\overline{\B}(m,n;r)$ contains precisely one board from each equivalence class of $\B(m,n;r)$, providing an optimal solution to the classification problem. 

\begin{proposition}\label{prop:even_rect_reduction}
    For $k,l\geq 1$ with $k\ne l$ and $r$ an odd integer, the set $\overline{\B}(2k,2l,r)$ contains exactly one board from each equivalence class of $\B(2k,2l,r)$ under $G=\cyc{H,V}$.
\end{proposition}

\begin{proof} 
By Theorem \ref{thm:recteven}, every equivalence class is represented at least once in $\overline{\B}(2k,2l;r)$, but two boards in the set $\overline{\B}(2k,2l;r)$ are equivalent only if they have the same board partition. Corollary \ref{cor:recteven} describes all cases in which non-identity symmetries preserve board partitions. Since $r=\sum\lambda_i$, each such case can only occur if $r$ is even.
\end{proof}

\begin{proposition}\label{prop:onerow}
    Let $l\geq 1$ and let $r$ be odd. Then the set $\overline{\B}(1,2l;r)$ contains exactly one board from each equivalence class of $\B(1,2l;r)$ under $G=\cyc{R_{180}}$.
\end{proposition}

\begin{proof} By Theorem \ref{thm:1byn} and Corollary \ref{cor:1byn}, two boards of $\overline{\B}(1,n;r)$ are equivalent under $G$ only if $\lambda_1=\lambda_2$, in which case $r$ is even.
\end{proof}

\begin{proposition}\label{prop:oneblocker}
    For $m,n\geq 1$ with $m\ne n$, the set $\overline{\B}(m,n;1)$ contains exactly one board from each equivalence class of $\B(m,n;1)$ under $G=\cyc{H,V}$.
\end{proposition}

\begin{proof}
    In each non-square case presented in Section \ref{sec:BPcases}, a board partition can only contain two equivalent boards if $\lambda_1=\lambda_i$ for some $i>1$. Since we must also have $\lambda_1\geq\lambda_i$ for all $i$, then $\lambda_1=1$ and so $\lambda_i=0$ for all $i>1$.
\end{proof}

To get a sense of how close we come to a canonical representation in general, we formally define the ratio we computed in our examples from $\S6$. Let 
\[\mathcal{R}(m,n;r):=|\overline{\B}(m,n;r)|/|\O_G(\B(m,n;r))|\]
for fixed $m$ and $n$ and $r$, where $G$ is the group of symmetries corresponding to $\B(m,n;r)$.  We note that $\mathcal{R}(m,n;r)=\mathcal{R}(m,n;mn-r)$. When computationally feasible, we can compute $\mathcal{R}(m,n;r)$ for all $r$ with $1\leq r\leq \lceil\frac{mn}{2}\rceil$, using the formulas for $|\O_G(\B(m,n;r))|$ from \S5. When $\mathcal{R}(m,n;r)=1$ then $\B(m,n;r)$ contains exactly one board from each equivalence class.

For even square boards of side-length $n\leq 22$, we find that $\mathcal{R}(n,n;r)<2$ except for $r=2$ and $r=4$, which are both less than $2.5$. In fact, the ratios appear to decrease as $r$ increases in each equivalence class of $r$ modulo 4. We make the following conjecture.

\begin{conjecture} For even $n\geq 8$ and $s\in\{0,1,2,3\}$, the finite sequence of ratios \\$\{\mathcal{R}(n,n;4j+s)\}$ is decreasing as $j$ increases for $0\leq 4j+s\leq \lceil\frac{n^2}{2}\rceil$ and $\mathcal{R}(n,n;r)\leq \mathcal{R}(n,n;2)$ for all $1\leq r\leq n^2$.
\end{conjecture}

\begin{rem} We have $\mathcal{R}(2,2;r)=1$ for all $1\leq r\leq 4$, $\mathcal{R}(4,4;r)\leq\mathcal{R}(4,4;4)\approx 2.19$ for all $1\leq r\leq 16$, and $\mathcal{R}(6,6;r)\leq\mathcal{R}(6,6;4)\approx 2.21$ for all $1\leq r\leq 36$.\end{rem}

For even rectangular boards of dimension up to $20\times 18$, we find that for even $r$, $\mathcal{R}(m,n;r)<1.75$, while $\mathcal{R}(m,n;r)=1$ for $r$ odd by Proposition \ref{prop:even_rect_reduction}. The ratios for even $r$ form a decreasing sequence except in two small cases. We make the following conjecture for boards of this type.

\begin{conjecture} For positive even $m$ and $n$ with $m\ne n$ and $(m,n)\notin\{(4,2),(6,2)\}$, the sequence of ratios $\{\mathcal{R}(m,n;2j)\}$ is decreasing as $j$ increases for $0\leq 2j\leq \lceil\frac{mn}{2}\rceil$ so $\mathcal{R}(m,n,r)\leq \mathcal{R}\mathcal(m,n;2)$ for all $1\leq r\leq mn$.
\end{conjecture}

\begin{rem} We have $\mathcal{R}(4,2;r)\leq\mathcal{R}(4,2;4)\approx 1.41$ for all $1\leq r\leq 8$ and $\mathcal{R}(6,2;r)\leq\mathcal{R}(6,2;4)\approx 1.47$ for all $1\leq r\leq 12$.\end{rem}

Due to longer run times, our computational evidence is not sufficient to make a formal conjecture for boards with odd side lengths, however preliminary computations seem to indicate similar behavior with the ratio maximized at $r=2$.

\bibliographystyle{plain}
\bibliography{References}

\end{document}